\documentclass[leqno,10pt]{amsart} 
\usepackage{amssymb, amscd, verbatim, amsmath, MnSymbol, stmaryrd}
\usepackage{color}

\numberwithin{equation}{section}
\newtheorem{thm}{Theorem}[section]
\newtheorem{lem}[thm]{Lemma}

\newtheorem{rem}[thm]{Remark}

\theoremstyle{definition}

\def\tr{\Delta}
\def\bea{\begin{eqnarray*}}
\def\eea{\end{eqnarray*}}
\def\be{\begin{eqnarray}}
\def\ee{\end{eqnarray}}
\def\a{\alpha}

\def\n{\nabla} \def\e{\epsilon} 
\def\o{\omega}\def\d{\delta}\def\vp{\varphi}

\begin{document}

 \title{Besse conjecture with positive isotropic curvature}

\author{Seungsu Hwang}
\address{Department of Mathematics\\ Chung-Ang University\\
84 HeukSeok-ro DongJak-gu \\ Seoul 06974, Republic of Korea.
}
\email{seungsu@cau.ac.kr}

 \author{Gabjin Yun$^*$}
\address{Department of Mathematics\\  Myong Ji University\\
116 Myongji-ro Cheoin-gu\\ Yongin, Gyeonggi 17058, Republic of Korea. }
\email{gabjin@mju.ac.kr} 

\thanks{$*$ Corresponding author}
\thanks{The first author was supported by the Basic Science Research Program through the National Research Foundation of Korea(NRF)  funded by the Ministry of Education(NRF-2016R1D1A1A09916749), and the second and corresponding author by the Ministry of Education(NRF-2019R1A2C1004948).}



 \maketitle

 \begin{abstract}
 The critical point equation arises as a critical point of the total scalar curvature functional 
defined on the space of constant scalar curvature metrics of a unit volume on a compact 
manifold. In this equation, there exists a function $f$ on the manifold  that satisfies the following
  $$
  (1+f){\rm Ric} = Ddf + \frac{nf +n-1}{n(n-1)}sg.
  $$
    It has been conjectured that if $(g, f)$ is a solution of the critical point equation, then $g$
   is Einstein and so $(M, g)$  is isometric to a standard sphere.
 In this paper, we show that this conjecture is true if the given Riemannian metric has positive isotropic curvature.
  \end{abstract}
 
  \setlength{\baselineskip}{15pt}
 
\section{introduction}

In this paper, we consider closed smooth manifolds admitting Riemannian metrics
 as the criticals of a curvature functional.
Let $M$ be a closed $n$-dimensional smooth manifold and ${\mathcal M}_1$ be the set of all smooth Riemannian metrics of a unit volume
on $M$. The total scalar curvature functional ${\mathcal S}$ on ${\mathcal M}_1$ is given by
$$
{\mathcal S}(g) = \int_M s_g\, dv_g.
$$
where $s_g$ is the scalar curvature of the metric $g \in {\mathcal M}_1$ and $dv_g$ denotes the volume form of $g$.
Critical points of ${\mathcal S}$ on ${\mathcal M}_1$ are known to be Einstein (cf. \cite{Be}).
Introducing a subset ${\mathcal C}$ of ${\mathcal M}_1$ consisting of metrics with a constant scalar curvature,
the Euler-Lagrange equation of ${\mathcal S}$ restricted to ${\mathcal C}$ can be written in the following
\bea
z_g=s_g'^*(f),
\eea
which is called the critical point equation (CPE in short). 
Here, $z_g$ is the traceless Ricci tensor defined by $z_g:= {\rm Ric}_g - \frac{s_g}{n}g$, 
${\rm Ric}_g$ is the Ricci tensor, and $s_g'^*(f)$ is given by 
$$ 
s_g'^*(f)=D_gdf -(\tr_g f)g -f {\rm Ric}_g,
$$
where $D_gdf$ and  $\tr_g f$ are  the Hessian  and (negative) Laplacian of $f$, respectively.
From the variational problem for curvature functional, the {\it Besse conjecture} \cite{Be} describes that a non-trivial solution of the CPE
\be
z_g = D_gdf -(\tr_g f)g -f {\rm Ric}_g,  \label{cpe}
\ee
should be Einstein. By taking the trace of (\ref{cpe}), we have
 \be 
 \tr _g f=-\frac{s_g}{n-1}f \label{cpet}
 \ee 
 and therefore, using $z_g = {\rm Ric}_g - \frac{s_g}{n}g$, (\ref{cpe}) can be rewritten as 
 \be
 (1+f)z_g = Ddf + \frac{s_gf}{n(n-1)}g.\label{eqn2020-5-7-2}
 \ee
Since $g$ is clearly Einstein when $f=0$, going forward, we assume that $f$ is not trivial. 
Besse conjecture describes that such a nontrivial solution $(g,f)$ of  (\ref{cpe}) is Einstein.
Notably, some progress has been made to this conjecture. For example, 
 if the Riemannian manifold is locally conformally flat,
    then the metric is Einstein \cite{laf}. It is also known that 
    Besse conjecture holds if the manifold has a harmonic curvature \cite{ych}, \cite{erra},  
or the metric is Bach-flat \cite{qy}. 
We say that $(M, g)$ has a harmonic curvature if the divergence of the Riemannian curvature tensor $R$ vanishes, in other words,  
$\d R = 0$, and  $(M, g)$ is said to be Bach-flat  
when the Bach tensor vanishes. Very recently, it was shown \cite{bal} that the Besse conjecture holds if the complete divergence 
of the Weyl curvature tensor $\mathcal W$ is free, ${\rm div}^4 {\mathcal W} = 0$ and the radially Weyl curvature vanishes.
If $\min_M f \ge -1$, it is clear 
 that $(M, g)$ is Einstein. In fact, if we let $i_{\nabla f}z_g =z_g(\nabla f, \cdot)$,  
 the divergence of $i_{\nabla f}z_g$ can be computed as ${\rm div}\left(i_{\nabla f}z_g\right) = (1+f)|z_g|^2$.
Hence, by integrating it over $M$, we have $z_g=0$ from the divergence theorem.
As a $1$-form or a vector field, the quantity $i_{\nabla f}z_g$ has a crucial structural meaning in CPE. In \cite{ych}, we show that
$i_{\n f}z_g$ is parallel to $\n f$ when $(M, g)$ has a harmonic curvature, and this property plays an important  role in proving the main theorem.

First, we prove the Besse conjecture when $z_g$ is vanishing in the direction $\n f$, which is, in fact, a generalization of the main result in \cite{ych}.

\begin{thm} \label{thm2020-8-22-3}
Let $(g,f)$ be a nontrivial solution of (\ref{cpe}). If $z(\nabla f, X)=0$ for any vector field
 $X$ orthogonal to $\nabla f$, then $(M,g)$ is isometric to a standard sphere.
\end{thm}

The second objective of this paper is to prove the Besse conjecture under the condition of positive isotropic curvature on $M$. 
Let $(M^n, g)$ be an $n$-dimensional Riemannian manifold with $n \ge 4$. 
The Riemannian metric $g = \langle \,\, ,\, \rangle$ can be extended either
to a {\it complex bilinear form} $(\,\, , \, )$ or a {\it Hermitian inner product} 
$\llangle\,\,, \, \rrangle$ on each complexified tangent space  $T_pM \otimes {\Bbb C}$ for $p \in M$.
A complex $2$-plane $\sigma \subset T_pM\otimes {\Bbb C}$ is {\it totally isotropic}
 if $ (Z, Z) = 0$ for any $Z \in \sigma$. For any $2$-plane $\sigma \subset T_pM\otimes {\Bbb C}$, we can define the
 complex sectional curvature of $\sigma$ with respect to $\llangle  \,  , \, \rrangle$ by
 \be
 {\rm K}_{\Bbb C}(\sigma) = \llangle {\mathcal R}(Z \wedge W), Z\wedge W\rrangle,
 \label{eqn2019-5-17-1}
 \ee
 where $\mathcal R : \Lambda^2 T_pM \to \Lambda^2 T_pM$ is the curvature operator and 
 $\{Z, W\}$ is a unitary basis for $\sigma$ with respect to $\llangle \,  , \, \rrangle$.
 
  A Riemannian $n$-manifold $(M^n, g)$ is said to have a
 {\it positive isotropic curvature} {\rm(PIC in short)} if
  the complex sectional curvature on the isotropic planes is positive, that is,
 for any totally isotropic $2$-plane $ \sigma \subset T_pM\otimes {\Bbb C}$,
 \be
 {\rm K}_{\Bbb C}(\sigma) >0. \label{eqn2018-4-22-1}
 \ee
 If $(M, g)$ has a positive curvature operator, then, it has a PIC  \cite{mm88}. 
Thus, a standard sphere $({\Bbb S}^n, g_0)$ has a PIC. 
Additionally, if the sectional curvature of $(M, g)$  is pointwise strictly quarter-pinched, then  
$(M, g)$ has a PIC \cite{mm88}.
  It is well-known that the product metric on ${\Bbb S}^{n-1}\times {\Bbb S}^1$ also has
  a PIC and the connected sum of manifolds with a PIC admits a PIC metric \cite{m-w}.
  The existence of Riemannian metrics with a PIC on compact manifolds which fiber
  over the circle is discussed in \cite{lab}. One of the main results on manifolds with a positive isotropic
  curvature is that, a simply connected compact $n$-dimensional Riemannian manifold with positive
  isotropic curvature is homeomorphic to a sphere \cite{mm88}. 
  In another major result relating the topology of a positive isotropic
  curvature manifold, it is proved \cite{mm88} that the homotopy groups $\pi_i(M) = 0$ for $2 \le i \le [\frac{n}{2}]$ when $\dim(M) = n$, and the fundamental group $\pi_1(M)$ cannot contain any subgroup
  isomorphic to the fundamental group of a closed orientable surface \cite{fra}, \cite{f-w}.
  For even-dimensional manifolds, it was proved \cite{m-w}, \cite{sea} that a PIC implies the vanishing of the second Betti number.
  On the contrary, a PIC implies that $g$ has a positive scalar curvature \cite{m-w}.
  More details about the PIC are provided in \cite{c-h} or \cite{ses} and the references are also presented therein.

 Our second main result can be stated as follows..

\begin{thm} \label{thm1}
Let $M$ be a compact $n$-dimensional smooth manifold with $n\geq 4$.  
If $(g,f)$ is a nontrivial solution of (\ref{cpe}) and $(M, g)$ has positive isotropic curvature, 
then $(M,g)$ is isometric to a standard sphere.
\end{thm}

Due to \cite{ych} and \cite{erra}, it suffices to prove that $(M,g)$ has harmonic curvature in Theorem~\ref{thm1}. 
Then, $g$ should be Einstein, which implies that $(M, g)$ is isometric to a standard sphere due to  Obata \cite{Ob}.
To show that $(M, g)$ has a harmonic curvature, we introduce a $2$-form, $df\wedge i_{\nabla f}z_g$, consisting of
the total differential, $df$, of the potential function $f$ and the traceless Ricci tensor
$z_g$  and prove that it vanishes when $(M, g)$ has positive isotropic curvature.



\vskip .5pc

\noindent
{\bf Notations: }  Hereafter, for convenience and simplicity, we denote curvatures 
${\rm Ric}_g, z_g, s_g$, and   the Hessian and Laplacian of $f$, $D_gdf, \Delta_g$  by
$r, z, s$, and $Ddf, \Delta$, respectively, if there is no ambiguity.
We also use the notation $\langle \,\, ,\,\, \rangle$ for metric $g$ or inner product induced by $g$ on tensor spaces.

\section{Some Preliminaries and tensors}

As a preliminary as well as for subsequent use, we recall the Cotton tensor and briefly describe the properties of 
the Cotton tensor. Additionally, we introduce a structural tensor $T$, which plays a key role in
proving our main theorems (For the definition of $T$, refer (\ref{defnt}) ). 
This structural tensor $T$ has deep relations to the critical point equation (\ref{cpe})
 and the Cotton tensor (cf. \cite{h-y}).

\vskip .5pc

\subsection{Cotton tensor}

Let $(M^n, g)$ be a Riemannian manifold of dimension $n$ with the Levi-Civita
 connection $D$. We begin with a differential operator acting on the space 
 of symmetric $2$-tensors. Let $b$ be a symmetric $2$-tensor on $M$. The
differential $d^Db$ can be defined as follows:
$$ 
d^Db(X, Y, Z) = D_Xb(Y, Z) - D_Yb(X, Z) 
$$
 for any vectors $X, Y$ and $Z$. The {\it Cotton tensor} 
 $C \in\Gamma(\Lambda^2 M \otimes T^*M)$ is defined by 
 \be 
 C = d^D \left(r - \frac{s}{2(n-1)} g\right) 
 = d^Dr - \frac{1}{2(n-1)} ds \wedge g. \label{eqn2017-4-1-1}
 \ee 
  It is known that, for $n=3$, 
$C=0$ if and only if $(M^3, g)$ is locally conformally flat.
 Moreover, for $n \ge 4$, the vanishing of the Weyl
tensor $\mathcal W$ implies the vanishing of the Cotton tensor $C$, while
$C=0$ corresponds to the Weyl tensor being harmonic, i.e, $\d \mathcal W
= 0$ due to the identity \cite{Be}: 
\be 
\d \mathcal W = - \frac{n-3}{n-2}d^D \left(r - \frac{s}{2(n-1)} g\right) 
= - \frac{n-3}{n-2} C\label{eqn2016-12-3-16} 
\ee 
under the following identification
$$ 
\Gamma(T^*M\otimes \Lambda^2M) \equiv \Gamma(\Lambda^2M \otimes T^*M). 
$$


Let $\{E_i\}$ be a local frame with $C(E_i, E_j, E_k) = C_{ijk}$. Then, it follows from (\ref{eqn2017-4-1-1}) that 
$$
C_{ijk} = R_{jk;i}-R_{ik;j} - \frac{1}{2(n-1)} S_i \d_{jk} + \frac{1}{2(n-1)} S_j \d_{ik}, 
$$ 
where $r(E_i, E_j) = R_{ij}, \, R_{ij;k} = D_{E_k}r(E_i, E_j)$ and $ds(E_i) = S_i.$
The first direct observation on $C$ is that the cyclic summation of $C_{ijk}$ is
 vanishing. The second observation on $C$ is that $C_{ijk}$ is
skew symmetric for the first two components, which implies that
$C(X, X, \cdot) = 0$ for any vector  $X$, and trace-free in any two indices.
 Since $\displaystyle{r = z + \frac{s}{n}g}$, equation  (\ref{eqn2017-4-1-1}) can be
  rewritten as 
 \bea
 C = d^D z + \frac{n-2}{2n(n-1)}ds \wedge g.\label{eqn2017-4-2-1} 
 \eea
 In the local coordinates, we have
 $$
C_{ijk} = z_{jk;i} - z_{ik;j} + \frac{n-2}{2n(n-1)} (S_i \d_{jk} - S_j \d_{ik}).
$$ 
In the CPE, since the scalar curvature $s$ is constant, we have
$$
C = d^D r = d^D z
$$
and
$$
C_{ijk}z_{jk;i}  =  |C|^2 + C_{ijk}z_{ik;j} = |C|^2 - C_{jik}z_{ik;j}.
$$ 
Here, we follow the Einstein convention for indices.
i.e., 
\bea
\frac{1}{2}|C|^2=   C_{ijk}z_{jk;i} = C(E_i, E_j, E_k)D_{E_i}z(E_j, E_k).\label{eqn2017-4-2-2} 
\eea
Moreover, 
\be 
\frac{1}{2}|C|^2 
&=& 
C_{ijk}z_{jk;i} = (C_{ijk}z_{jk})_{;i} - C_{ijk;i}z_{jk} \nonumber\\ 
&=&
(C_{ijk}z_{jk})_{;i} + \langle \d C, z\rangle,\label{eqn2018-1-23-1} 
\ee
where $\d C (X, Y) = -{\rm div}C(X, Y)  = -D_{E_i} C(E_i, X, Y)$.

\vspace{.13in}
\subsection{ The  tensor $T$}



We now define a $3$-tensor $T$ by  
\be
T= \frac 1{n-2}\, df  \wedge z +\frac 1{(n-1)(n-2)}\, i_{\nabla f }z \wedge g. \label{defnt}
\ee
Recall that $i_{\n f}$ denotes the (usual) interior product to the first factor defined by
$i_{\n f}z(X)  = z(\n f, X)$ for any vector $X$.
As the Cotton tensor $C$, the cyclic summation of $T_{ijk}$ vanishes.
It is also easy to see that the trace of $T$ in any two summands vanishes. In fact, since $T(X, Y, Z) = - T(Y, X, Z)$, we have,
${\rm tr}_{12} T = 0$, and from the definition of $T$ together with ${\rm tr}(z) = 0$, we can show ${\rm tr}_{13}T = 0$ directly. 
Finally, since the cyclic summation of $T$ vanishes, we have
${\rm tr}_{23} T = -{\rm tr}_{12}T - {\rm tr}_{13}T =  0$.
The tensor $T$ looks similar to the tensor in \cite{CC} (cf. \cite{k-o}) 
where the authors used this in
 classifying the complete Bach flat gradient shrinking Ricci solitons.

The first relation of the Cotton tensor $C$ to the tensor $T$ is the following
for the CPE.

\begin{lem}\label{lem2019-5-23-10}
Let $(M^n, g, f)$ be a non-trivial solution of the CPE. Then
\bea
(1+f) C = {\tilde i}_{\n f} \mathcal W - (n-1)T.\label{eqn2017-6-12-10-1}
\eea
Here $\tilde{i}_{X}$ is the interior product to the last factor defined by
$\tilde{i}_{X}{\mathcal W} (Y, Z, U)= {\mathcal W}(Y, Z, U, X)$ for any vectors
$ Y, Z$, and $U$.
\end{lem}
\begin{proof}
Taking the differential $d^D$ in (\ref{cpe}), we have
\bea
(1+f)d^D z = {\tilde i}_{\n f}\mathcal W - \frac{n-1}{n-2}df\wedge z -
\frac{1}{n-2}i_{\n f}z \wedge g.\label{eqn2016-12-3-10-1} 
\eea
(For more details about this, refer \cite{ych}). 
 From the definition of $T$ together with $C = d^Dz$, we obtain
\bea
(1+f) C = {\tilde i}_{\n f} \mathcal W - (n-1)T.
\eea
\end{proof}

\vspace{.13in}
Now, we derive certain expressions relating $T$ and $C$  to their divergences. 
For a symmetric $2$-tensor $b$ on a Riemannian manifold $(M^n, g)$, 
we define ${\mathring {\mathcal W}}(b)$ by
$$ 
{\mathring {\mathcal W}}(b)(X, Y) = \sum_{i=1}^n b(\mathcal W (X, E_i)Y, E_i) 
= \sum_{i, j=1}^n \mathcal W (X, E_i, Y, E_j)b(E_i, E_j) 
$$ 
for any local frame $\{E_i\}$. For the Riemannian curvature tensor $R$,
${\mathring R}(b)$ is similarly defined.


\begin{lem}\label{lem2018-2-10-1} 
Let $(g, f)$ be a solution of the CPE. Then 
\bea
\d ({\tilde i}_{\n f} \mathcal W) = - \frac{n-3}{n-2} \widehat C 
+ (1+f) {\mathring {\mathcal W}}(z), 
\label{eqn2018-2-10-1}
\eea
where ${\widehat C}$ is a $2$-tensor defined as 
$$ 
{\widehat C}(X, Y) = C(Y, \n f, X) 
$$ 
for any vectors $X, Y$. 
\end{lem} 
\begin{proof} 
Let $\{E_i\}$ be a local frame normal at a point in $M$. At the point, it follows from definition
 together with the CPE that
\bea 
\d ({\tilde i}_{\n f} \mathcal W) (X, Y) = \d \mathcal W(X, Y, \n f) - Ddf(E_i, E_j) 
\mathcal W (E_i, X, Y, E_j).
\eea
Since the trace of $\mathcal W$ in the first and fourth components  vanishes, applying
(\ref{eqn2016-12-3-16}), we obtain
$$
\d \mathcal W(X, Y, \n f) - Ddf(E_i, E_j) \mathcal W (E_i, X, Y, E_j)
=  - \frac{n-3}{n-2}C(Y, \n f, X) + (1+f) {\mathring {\mathcal W}}(z) (X, Y). 
$$
\end{proof}

Moreover, the divergence of the tensor ${\widehat C}$ has 
the following form.

\begin{lem}\label{lem2018-2-14-10} 
For any vector $X$, we have 
\bea
\d {\widehat C} (X) = -(1+f)\langle i_X C, z\rangle.\label{eqn2018-2-9-2-10} 
\eea
\end{lem} 
\begin{proof} 
Let $\{E_i\}$ be a local geodesic frame around a point $p\in M$.
 Using the fact that $\d C$ is symmetric and  
 that cyclic summation of $C$ is vanishing, we can show
$$
  \d {\widehat C} (E_j) = - (1+f)z(E_i, E_k) C(E_j, E_k, E_i) 
  = - (1+f)\langle i_{E_j}C, z\rangle. 
$$
\end{proof}

We are now ready to prove that $T=0$ implies the CPE conjecture.

\begin{thm}\label{thm2018-1-20-11}
 Let $(g, f)$ be a non-trivial solution of the CPE. 
 If $T = 0$, then the Besse conjecture holds.
\end{thm}
\begin{proof} 
It suffices to prove that $C=0$. First, on each level hypersurface $f^{-1}(t)$ for a regular value $t$ of $f$,
it follows from the definition of $T$ that
\bea 
i_{\nabla f}T &=& \frac {|\nabla f|^2}{n-2}\left( z+\frac {\a}{n-1}g\right).
\eea
Since $T=0$, we have
\be 
z(E_i, E_j)= -\frac {\a}{n-1} \delta_{ij} \quad (2\leq i,j\leq n)\label{eqt08}
\ee
for a local orthonormal frame $\{E_1=N,\cdots, E_{n-1}, E_n\}$ with $N = \frac{\n f}{|\n f|}$. Also, from
$$ 
0=(n-2)i_{\nabla f}T(E_j, \nabla f)= \frac {n-2}{n-1}|\nabla f|^2 z(E_j, \nabla f)
$$
for $2\leq j\leq n$, we have
\be 
z(\nabla f, E_j)=0.\label{eqt09}
\ee
Therefore, 
\be 
\qquad \langle i_{\nabla f}C, z\rangle = -\frac {\a}{n-1} \sum_{i=2}^{n}
C(\nabla f, E_i, E_i)=\frac {\a}{n-1} C(\nabla f, N, N)=0. \label{eqy2}
\ee
On the other hand, from Lemma~\ref{lem2019-5-23-10},  we have
\be 
(1+f)C = {\tilde i}_{\n f} {\mathcal W} \label{eqn2018-2-10-7} 
\ee 
and hence, 
\bea
C(X, Y, \n f) = 0 \label{eqn2018-2-10-5} 
\eea
for any vector fields $X$ and $Y$. 
Since the cyclic summation of $C$ vanishes, we have 
$$ 
C(Y, \n f, X) + C(\n f, X, Y) = 0.
$$ 
Taking the divergence $\d$ of (\ref{eqn2018-2-10-7}) and applying 
Lemma~\ref{lem2018-2-10-1}, we obtain
\bea 
-i_{\n f} C + (1+f)\d C = \frac{n-3}{n-2} i_{\n f} C + (1+f) {\mathring {\mathcal W}}(z). 
\eea 
Thus, 
\be 
(1+f)\d C - (1+f) {\mathring {\mathcal W}}z = \frac{2n-5}{n-2}i_{\n f}C.\label{eqn2018-2-10-8} 
\ee
Note that,  by (\ref{eqn2018-2-10-7}) and definition of $\mathring{\mathcal W}(z)$, the following
$$ 
\mathring{\mathcal W}(z) (\nabla f, X)= -(1+f) \langle i_XC, z\rangle
$$
 holds for any vector field $X$. In particular, by (\ref{eqy2}), we have
$$
\mathring{\mathcal W}(z) (\nabla f, \n f) =0.
$$
Consequently, by  (\ref{eqn2018-2-10-8}), we obtain
\be
(1+f)\d C(N, N) = 0.\label{eqn2018-2-10-10} 
\ee 
Finally, it follows from (\ref{eqt08}) and (\ref{eqt09}) together with  
$\displaystyle{\sum_{j=2}^{n}C_{ijj;i} = - C_{i11;i}}$ that
\bea 
\langle \d C, z\rangle &=& - C_{ijk;i}z_{jk} =
 \frac{\a}{n-1} \sum_{i=1}^n \sum_{j=2}^{n}C_{ijj;i} - \a \sum_{i=1}^n  C_{i11;i}\\
&=&
-\frac{n\a}{n-1} \sum_{i=1}^n  C_{i11:i} = \frac{n\a}{n-1} \d C(N, N) =0. 
\eea 
Hence, integrating (\ref{eqn2018-1-23-1}) over $M$, we have $C = d^D z = 0$, 
which implies that $(M, g)$ has a harmonic curvature. Therefore, it follows from 
\cite{ych} and \cite{erra} that $(M, g)$ is  Einstein.
\end{proof}

\begin{rem}
{\rm
Introducing the Bach tensor on a Riemannian manifold $(M^n, g)$,
 we can directly show that $C=0$ in Theorem~\ref{thm2018-1-20-11} without using
  the divergence theorem in a higher dimensional case $n \ge 5$. 
  The {\it Bach tensor} $B$ is 
  defined as 
\bea
B = \frac{1}{n-3}\d^D \d \mathcal W + \frac{1}{n-2}{\mathring {\mathcal W}}(r).\label{eqn2018-2-11-2} 
\eea
When the scalar curvature $s$ is constant, 
we have ${\mathring W}(r) = {\mathring W}(z)$. Therefore, from (\ref{eqn2016-12-3-16}), we have 
\bea
B = \frac{1}{n-2}\left(-\d C + {\mathring {\mathcal W}}(z)\right).\label{bach302} 
\eea
Furthermore, the following property holds in general (cf. \cite{CC}) for $n \ge 4$:
for any vector field $X$,  
$$ 
(n-2)\d B (X) = - \frac{n-4}{n-2} \langle i_X C, z\rangle. 
$$ 
The complete divergence of the Bach tensor has the following form:
\bea
\d\d B = \frac{n-4}{(n-2)^2} \left(\frac{1}{2} |C|^2 - \langle \d C, z\rangle \right).\label{eqn2018-2-22-1} 
\eea
Using these identities on the Bach tensor and taking the divergence of $T$, the following identities can be obtained:
\begin{itemize}
\item[(i)] $ (n-2)(1+f)B = -i_{\n f}C + \frac{n-3}{n-2}{\widehat C} + (n-1) \d T.$
\item[(ii)] For any vector field $X$,  
\bea
\d\d T (X) = \frac{1}{n-2}(1+f)\langle i_X C, z\rangle + \langle i_X T,
z\rangle.\label{eqn2018-2-9-1} 
\eea 
\end{itemize}
Therefore, we obtain the following:
Let $(g, f)$ be a non-trivial solution of the CPE and $T = 0$. Then
$\langle i_X C, z \rangle = 0$ for any vector field $X$ by (ii), and hence,
$$
\frac{1}{2}|C|^2 = \langle \d C, z\rangle =0.
$$
}
\end{rem}

Before closing this section, we would like to mention some of the identities on the square norm 
and the divergence of $T$, which will be used later in the proof of our main theorem.

\begin{lem}{\rm (\cite{GH1})}\label{lem2019-5-28-5}
Let $(g, f)$ be a  solution of the CPE, Then,
$$ 
|T|^2= \frac 2{(n-2)^2}|\nabla f|^2 \left( |z|^2-\frac n{n-1}|i_Nz|^2\right).
$$
\end{lem}

\begin{lem} \label{lem2020-5-7-10}
Let $(g, f)$ be a  solution of the CPE, Then,
$$
(n-1)\d T = \frac{sf}{n-1}z - D_{\n f}z + \frac{1}{n-2}\widehat C + \frac{n}{n-2} (1+f)z\circ z - \frac{1}{n-2}(1+f) |z|^2 g. 
$$
Here $ z\circ z$ is defined by
$$
 z\circ z(X, Y) = \sum_{i=1}^n z(X, E_i) z(Y, E_i)
 $$
 for a local frame $\{E_i\}_{i=1}^n$ of $M$.
\end{lem}
\begin{proof} 
By using (\ref{cpet}) and (\ref{eqn2020-5-7-2}), we can compute
\bea
{\d} (df\wedge z) =  \frac{sf}{n}z - D_{\n f}z + (1+f) z\circ z.\label{eqn2018-9-10-1}
\eea
From the fact that $\delta z = 0$ and using (\ref{eqn2020-5-7-2}) again, we can obtain
\bea
{\d} (i_{\n f}z \wedge g) =  D_{\n f}z + {\widehat C} - (1+f)|z|^2 g + (1+f) z\circ z 
- \frac{sf}{n(n-1)} z. \label{eqn2018-9-10-2}
\eea
By combining these identities, the proof follows from the definition of the tensor $T$.
\end{proof}

\section{CPE with $z(\n f, X)=0$}

For the potential function $f$ of the CPE, we let
$$
\a := z(N, N)
$$
for convenience. Here recall that $N = \frac{\n f}{|\n f|}$.
Note that the function $\a$ is well-defined only on the set $M \setminus {\rm Crit}(f)$, 
where ${\rm Crit}(f)$ is the set of
all critical points of $f$. However, $|\a| \le |z|$, $\a$ can be extended to a $C^0$ function on the whole $M$. 
Refer \cite{ych} for more details.

\noindent
First, it is easy to compute that
\be
N(|\n f|) = (1+f)\a - \frac{sf}{n(n-1)}\label{eqn2020-5-7-3}
\ee
and thus,
\be
N\left(\frac{1}{|\n f|}\right)  = - \frac{1}{|\n f|^2} \left[(1+f)\a - \frac{sf}{n(n-1)}\right].\label{eqn2020-5-7-4}
\ee
In particular, we have the following lemma.

\begin{lem}\label{lem2020-8-12-1}
Assume $z(\nabla f, X)=0$ for any vector field  $X$ orthogonal to $\nabla f$. 
Then $i_{\n f} z = \a \n f$ as a vector field, and $D_NN = 0.$ 
\end{lem}
\begin{proof}
It is obvious that $i_{\n f} z = \a \n f$. From this, by combining (\ref{eqn2020-5-7-2}) and (\ref{eqn2020-5-7-3}), 
it is easy to see that $D_NN = 0$. 
\end{proof}

From now, throughout this section, we assume that $(M^n, g, f)$ is a nontrivial solution of the CPE satisfying
\be
z(\n f, X) = 0 \label{eqn2020-5-7-1}
\ee
for any vector field $X$ orthogonal to $\n f$. In this case, by Lemma~\ref{lem2020-8-12-1}, we can write
\be
i_{\n f} z = \a df\label{eqn2020-8-22-1}
\ee
as a $1$-form. As the interior product $\tilde i_{\n f }{\mathcal W}$ of $4$-tensor, 
we define the interior product $\tilde i_{\n f}T$ and $\tilde i _{\n f}C$  as
$$
\tilde i_{\n f}T(X, Y) = T(X, Y, \n f)\quad \mbox{and}\quad \tilde i_{\n f}C(X, Y) = C(X, Y, \n f)
$$
for vector fields $X$ and $Y$.

\begin{lem}\label{Cotz}
Assume $z(\nabla f, X)=0$ for $X \perp \n f$. Then
$$
\tilde{i}_{\nabla f}T=0 \quad \mbox{and}\quad \tilde{i}_{\nabla f}C=0$$
Additionally, we have
\be
(1+f)|z|^2= \nabla f(\a)  -\frac {sf}{n-1}\a. \label{eqn2}
\ee
\end{lem}
\begin{proof}
 By definition of $T$, we have
 $$
 T(X, Y, \nabla f) =\frac {1}{n-1}\, df\wedge i_{\nabla f} z(X,Y)
$$
for any vectors $X$ and $Y$. Since $i_{\n f} z = \a df$, we obviously have
$df \wedge i_{\n f} z = 0$, which shows  $\tilde{i}_{\nabla f}T=0$. 
The second equality follows from Lemma~\ref{lem2019-5-23-10}.
For (\ref{eqn2}), we note from the CPE that
$$
\d (i_{\n f}z) = -(1+f)|z|^2.
$$
Hence, we can obtain (\ref{eqn2}) by taking the divergence of (\ref{eqn2020-8-22-1}).
\end{proof}

For a nontrivial solution $(g, f)$ of the CPE satisfying $ z(\n f, X) = 0$ for $X \perp \n f$, 
there are two important properties that together play a key role in proving our main theorems.
One is that both the functions, $\a = z(N, N)$ and $|z|^2$, are constants on each level set $f^{-1}(c)$, and the other is that there are no critical points of $f$ except its maximum and minimum
 points in $M$.

\begin{lem}\label{lem2020-8-23-2}
The functions $|\n f|, \a$ and $|z|^2$ are all constants on each level set $f^{-1}(c)$ of $f$.
\end{lem}
\begin{proof}
We may assume that $c$ is a regular value of $f$ by the Sard's theorem. 
Let $X$ be a tangent vector on $f^{-1}(c)$. Then, from (\ref{eqn2020-5-7-2}), we have
$$
\frac{1}{2} X(|\n f|^2) = Ddf(\n f, X) = (1+f)z(\n f, X) - \frac{sf}{n(n-1)} g(\n f, X) = 0.
$$ 
Now, let $X$ be a (local) vector field orthogonal to $\nabla f$ such that $z(X, \n f ) = 0$.
Since $D_NN=0$ by Lemma~\ref{lem2020-8-12-1}, we have
$g(D_NX, N) = - g(X, D_NN) = 0$ and $D_N\n f = - |\n f|N\left(\frac{1}{|\n f|}\right) \n f$. So,
\be
D_N z(X, \n f) = - z(D_N X, \n f)-z(X, D_N \n f) = 0.\label{eqn2020-8-26-1}
\ee
Since $\tilde{i}_{\nabla f}C=0$ by Lemma~\ref{Cotz}, we have
$$
0=C(X, N, \nabla f)=D_X z(N, \nabla f)-D_{N}z(X, \nabla f) = D_X z(N, \nabla f)=|\nabla f|X(\a),
$$
implying that $\a $ is a constant on $f^{-1}(c)$.

The property, $D_NN=0$, also implies $ [X, N]$ is orthogonal to $\n f$ and therefore,
$$
X(N(\alpha)) = N(X(\alpha)) - [X, N](\alpha) = 0.
 $$
Since $\n \alpha = N(\alpha)N$, it shows $Dd\a(X, \n f) = 0$, and from (\ref{eqn2}), we have $X(|z|^2) = 0$.
Consequently, we conclude that $|z|^2$ is a constant along each level set $f^{-1}(c)$ of $f$.
\end{proof}

\begin{lem} \label{lem2020-8-26-2}
$\a$ is nonpositive on $M$.
\end{lem}
\begin{proof}
First assume that each level set $f^{-1}(t)$ is connected.
Suppose that $\a >0$ on a level set $f^{-1}(c)$.
If $c \le -1$, the divergence theorem shows that
$$
0\geq \int_{f \le c} (1+f)|z|^2= \int_{f \le c} {\rm div}(i_{\n f}z) = \int_{f=c} \a |\nabla f|, 
$$
which is impossible as $\a$ is constant on each level set. 
Note that $-1$ is a regular value of $f$ (see Appendix).
 If $c>-1$, the divergence theorem states that
$$
0\le  \int_{f \ge c} (1+f)|z|^2= - \int_{f=c} \a |\nabla f|, 
$$
which is also impossible. In any case, we get a contradiction.
For general case, i.e., if level sets may not be connected, refer \cite{ych} 
(In the proof of Lemma 5.2 there,  we used only the CPE and the condition $z(\n f, X) = 0$ for $X \perp \n f$).

\end{proof}

\begin{lem}\label{cor191} 
Let $(g,f)$ be a non-trivial solution of (\ref{cpe}) on an $n$-dimensional compact manifold $M$ 
with $z(\n f, X) = 0$ for $X \perp \n f$. Then, there are no critical points 
of $f$ except at its maximum and minimum points unless $g$ is an Einstein.
\end{lem}
\begin{proof}
Note that $|\nabla f|^2$ is constant on each level set of $f$.  
From the Bochner-Weitzenb\"ock formula together with (\ref{eqn2020-5-7-2}),   
we have
$$ 
\frac 12 \tr |\nabla f|^2-\frac 12 \frac {\nabla f(|\nabla f|^2)}{1+f} 
+\frac s{n(n-1)}\frac 1{1+f}|\nabla f|^2=|Ddf|^2.
$$
By  maximum principle,  the function $|\nabla f|^2$ cannot have its local maximum in 
$M_0:=\{x\in M \, |\, f(x) < -1\}$.
Let $p$ be a critical point of $f$ in $M$ other than the minimum or maximum points 
of $f$. From the argument above, we can see that 
 $p$ cannot be in $M_0:=\{x\in M \, |\, f(x) < -1\}$ unless $f$ is a constant.
 In fact, if $p \in M_0$ and $p$ is not a minimum point, then there should a local maximum point of $|\nabla f|^2$ in $\{x\in M\,|\, f < f(p)\}$,  which is impossible.
 We also claim that $p$ cannot be in $M^0:=\{x\in M \, |\, f(x) > -1\}$. 
 Recall that we have $f(p) \ne -1$ and
$$ 
\nabla f(|\nabla f|^2)= \frac {2s}{n(n-1)}|\nabla f|^2>0
$$
on the set $f=-1$. This shows that  $|\nabla f|^2 $ might have its maximum at a point, say $q$, in $M^0$. However,  since
$$ 
0 =\frac 12 \nabla f(|\nabla f|^2)=(1+f) \a |\nabla f|^2 -\frac {sf}{n(n-1)}|\nabla f|^2
$$
at the point $q$ and $(1+f)\a |\nabla f|^2 \le 0$, $q$  should lie in 
$\Omega:=\{x\in M \, |\, -1<f(x)<0\}$.
This implies that there are no critical points of $f$ in $M^0\setminus \Omega$.
Now, if $|z|^2(p)\neq 0$ with $p \in \Omega$, by Lemma 3.9 of \cite{ych},
 the point $p$ has to be a local maximum point of $f$, which is not possible because 
$$ 
\tr f=-\frac s{n-1}f>0
$$
on $\Omega$. Thus, we have $|z|^2(p) = 0$ with $p \in \Omega$. In particular,
if we let $f(p) = c$ with $-1< c <0$, then we have $\a =0$ on the level set $f^{-1}(c)$. 
Finally, from the divergence theorem, we have
\bea
0 \le   \int_{f >c}(1+f)|z|^2 &=& - \int_{f=c} \a |\n f| = 0,
\eea
which means $z=0$ on the set $f \ge c$. Then, by the analyticity of $g$ and $f$, 
$z$ must be vanishing on the set $M^0$ and consequently on the entire $M$.
\end{proof}
 
Furthermore, we can show that the potential function $f$ has only one maximum point and only one minimum point when $(M, g)$ has positive isotropic curvature.

\begin{thm}\label{thm2019-12-23-1}
Let $(g,f)$ be a non-trivial solution of the CPE on an $n$-dimensional compact manifold $M$ 
with $z(\n f, X) = 0$ for $X \perp \n f$. Then, there are  only two isolated critical points of $f$, in other words, only one maximum point and only one minimum point of $f$ on $M$ exist.
\end{thm}
\begin{proof}
Let ${\min_M f = a}$ and suppose $f^{-1}(a)$ is not discrete. By Lemma~\ref{cor191}, 
all the connected components of the level hypersurface $f^{-1}(t)$ for any regular value 
$t$ of $f$ have the same topological type.
In particular, since $M$ is smooth and $f^{-1}(a)$ is not discrete, $f^{-1}(a)$
 must be a hypersurface and it also has the same topological type as any
  connected component of $f^{-1}(t)$ for any regular value $t$ of $f$.
  Moreover, for a sufficiently small $\epsilon >0$, $f^{-1}(a+\epsilon)$ has two
  connected components, say $\Sigma_\epsilon^+, \,\, \Sigma_\epsilon^-$.
  Note that $\Sigma_0^+=\Sigma_0^- = f^{-1}(a)$.
  
  Let $\nu$ be a unit normal vector field on $\Sigma:= f^{-1}(a)$. Then, 
  $\nu$ can be extended smoothly to a vector field $\Xi$ defined on a tubular neighborhood 
  of $f^{-1}(a)$ such that   $\Xi|_{f^{-1}(a)} = \nu$ and 
  $\Xi|_{\Sigma_\epsilon^+} =  N = \frac{\n f}{|\n f|}$, 
  $\Xi|_{\Sigma_\epsilon^-} =  -N = -\frac{\n f}{|\n f|}$.
 Note that
 $$
 \lim_{\epsilon\to 0+} N = \lim_{\epsilon \to 0-}(-N)  = \nu.
 $$
  On the hypersurface $f^{-1}(a+\epsilon)$ near $f^{-1}(a)$, the Laplacian of $f$ is given by
  $$
  \Delta f = Ddf(N, N) + m |\n f|,
  $$
  where   $m$ denotes the mean curvature of $f^{-1}(a+\epsilon)$. 
  In particular, by letting $\e \to 0$, we have
  $$
  \Delta f =   Ddf(\nu, \nu)
  $$
  on $f^{-1}(a)$. Note that the mean curvature $m$ does not blow up on $f^{-1}(a)$.
  Therefore, from the CPE\, 
  $  -\frac{sf}{n-1} = (1+f)z(\nu, \nu) - \frac{sf}{n(n-1)}$, we have
\be
 z(\nu, \nu) = - \frac{sf}{n(1+f)} = - \frac{sa}{n(1+a)} \label{eqn2019-1210-1}
  \ee
on the set $\Sigma = f^{-1}(a)$. In particular, we have $z_p(\nu, \nu) <0$ , since we may assume $a<-1$ as mentioned in Introduction.
  
  Now, since $a$ is the minimum value of $f$, for
each point $p \in f^{-1}(a)$, the index of $Ddf_p$ is zero, i.e., for any vector
$v $ at $p$, we have $Ddf_p(v, v) \ge 0$.
Choosing an orthonormal basis $\{e_1 = \nu(p), e_2, \cdots, e_n\}$ on $T_pM$, we obtain
$$
Ddf_p(\nu, \nu ) = (1+f)z_p(\nu, \nu) - \frac{sf}{n(n-1)}\ge 0
$$
and for all $2\le i \le n$,
$$
Ddf_p(e_i, e_i ) = (1+f)z_p(e_i, e_i) - \frac{sf}{n(n-1)}\ge 0.
$$
In particular, from (\ref{eqn2019-1210-1}),
$$
z_p(e_i, e_i) \le \frac{s}{n(n-1)}\cdot \frac{a}{1+a} = - \frac{1}{n-1} z_p(\nu, \nu).
$$
Summing up $i=2, \cdots, n$, we can see
\be
z_p(e_i, e_i) = - \frac{1}{n-1} z_p(\nu, \nu) = \frac{s}{n(n-1)}\cdot\frac{a}{1+a} >0\label{eqn2021-1-25-1}
\ee
 for each  $i=2, \cdots, n$.
Thus, for each $i =2, \cdots, n$,
\be
Ddf_p(e_i, e_i) = (1+f)z_p(e_i, e_i) - \frac{sf}{n(n-1)} = 0.\label{eqn2021-1-24-1}
\ee

 \vspace{.2in}
Now, we claim that the minimum set  $\Sigma = f^{-1}(a)$ is totally geodesic  and in particular, the mean curvature is vanishing, i.e., $m=0$ on $\Sigma$.
In fact, fix $i$ for $i=2,3, \cdots, n$ and  let $\gamma : [0, l) \to M$ be a unit speed geodesic such that $\gamma(0) =p\in \Sigma$
and $\gamma'(0) = e_i \in T_p\Sigma$.
Then we have 
\be
D_{\gamma'} N &=&
\langle \gamma', N\rangle N\left(\frac{1}{|\n f|}\right)\n f + \frac{1}{|\n f|}\left[(1+ f) z(\gamma', \cdot) - \frac{sf}{n(n-1)}\gamma'\right].
\label{eqn2020-6-17-1}
\ee
Recall that, by (\ref{eqn2020-5-7-4})
$$
N\left(\frac{1}{|\n f|}\right)  = - \frac{1}{|\n f|^2} \left((1+ f)\alpha - \frac{sf}{n(n-1)}\right).
$$
Let 
$$
\gamma' = \langle \gamma', N\rangle N + (\gamma')^{\top},
$$
where $\gamma'(t)^{\top}$ is the tangential component of $\gamma'(t)$ to $f^{-1}(\gamma(t))$, 
and substituting these into (\ref{eqn2020-6-17-1}), we obtain
$$
|\n f| D_{\gamma'} N =   (1+ f) z((\gamma')^\top, \cdot)  - \frac{sf}{n(n-1)}(\gamma')^\top.
$$
Taking the covariant derivative in the direction $N$, we have
\bea
&&Ddf(N, N) D_{\gamma'}N + |\n f|D_ND_{\gamma'}N \\
&&\qquad
= 
|\n f| z((\gamma')^\top, \cdot)  + (1+f) D_N[z((\gamma')^{\top}, \cdot)] -  \frac{s}{n(n-1)}|\n f| (\gamma')^\top - \frac{sf}{n(n-1)}D_N \gamma'^{\top}.
\eea
Letting $t\to 0+$, we obtain
\be
-\frac{sa}{n-1} D_{e_i}\nu = (1+a) D_{\nu} [z((\gamma')^{\top}, \cdot)]\bigg|_p - \frac{sa}{n(n-1)} D_{\nu} (\gamma')^{\top}\bigg|_p\label{eqn2021-1-25-2}
\ee
because the covariant derivative depends only on the point $p$ and initial vector $e_i$.
Now since $z(N, X) = 0$ for $X \perp N$, we may assume that $\{e_i\}_{i=2}^n$ diagonalizes $z$ at the point $p$. Then
\be
z(D_{\nu} (\gamma')^{\top}, e_i)\bigg|_p = 0.\label{eqn2021-1-25-3}
\ee

\vspace{.15in}

{\bf Assertion:}\, $\nu(z(\gamma'(t), \gamma'(t))|_{t=0} = 0.$

\vspace{.15in}
\noindent
Defining $\vp(t):= f\circ \gamma(t)$, we have  $\vp'(0) =0$ and also $\vp''(0) = Ddf(e_i, e_i) = 0$ by (\ref{eqn2021-1-24-1}).
Since $\Sigma$ is the minimum set of $f$, $\vp'(t)$ is nondecreasing when $\vp(t)$ is sufficiently close to $a = \min f$ and
so $\vp''(t) \ge 0$ for sufficiently small $t$. So
\be
\vp''(t) = Ddf(\gamma'(t). \gamma'(t)) = [1+\vp(t) ] z(\gamma'(t), \gamma'(t)) - \frac{s}{n(n-1)}\vp(t) \ge 0\label{eqn2021-1-26-1}
\ee
for sufficiently small $t>0$.  However, by (\ref{eqn2021-1-25-1}), we have $z(\gamma'(t), \gamma'(t)) >0$ for sufficiently small $t>0$.
So,
$$
0< z(\gamma'(t), \gamma'(t)) \le \frac{s}{n(n-1)}\cdot \frac{\vp(t)}{1+\vp(t)}
$$
for sufficiently small $t>0$. Defining $\xi(t) = z(\gamma'(t), \gamma'(t)) - \frac{s}{n(n-1)}\cdot \frac{\vp(t)}{1+\vp(t)}$, we have
$\xi(t) \le 0$ and $\xi(0) = 0$ by  (\ref{eqn2021-1-24-1}). Thus,
$$
\xi'(0)  = \frac{d}{dt}\bigg|_{t=0} z(\gamma'(t), \gamma'(t)) \le 0.
$$
Now considering an extension of $\gamma$ to an interval $(-\e, 0]$, 
 we can see that $\frac{d}{dt}|_{t=0} z(\gamma'(t), \gamma'(t))$ cannot be negative since $\Sigma$ is the minimum set of $f$. 
In other words, we must have
$$
\frac{d}{dt}\bigg|_{t=0} z(\gamma'(t), \gamma'(t)) = 0,
$$
which completes the Assertion.

\vspace{.12in}
Let $\{N, E_2, \cdots, E_n\}$ be a local frame around $p$ such that $E_i(p) = e_i$ for $2 \le i$. Then, for $i \ge 2$,
$$
z(E_i, \cdot) = \sum_{j=2}^n z(E_i, E_j) E_j
$$
and
$$
D_N[z(\gamma'^\top, \cdot)] = D_N \left[z(\gamma'^\top, E_j)E_j\right] = N\left(z(\gamma'^\top, E_j)\right)E_j + z(\gamma'^\top, E_j)D_NE_j.
$$
So,
$$
\langle D_N[z(\gamma'^\top, \cdot)], E_i\rangle  = N\left(z(\gamma'^\top, E_i)\right) + z(\gamma'^\top, E_j) \langle D_NE_j, E_i\rangle.
$$
Note that $z(\gamma', \gamma') = \Vert \gamma'^\top\Vert z(\gamma'^{\top}, E_i) + \langle \gamma', N\rangle^2 \a$ with $\a = z(N, N).$
Since $\Vert \gamma'^\top\Vert$ attains its maximum at $p$ and $\langle \gamma', N\rangle^2$ attains its minimum $0$ at $p$, we have
$$
\nu [z(\gamma', \gamma')]\bigg|_p =  \Vert \gamma'^\top\Vert \nu[z(\gamma'^{\top}, E_i)] + \langle \gamma', N\rangle^2 \nu(\a) = 
\nu[z(\gamma'^{\top}, E_i)]\bigg|_p,
$$
which shows that
\be
\nu[z(\gamma'^{\top}, E_i)]\bigg|_p = 0\label{eqn2021-2-4-1}
\ee
by {\bf Assertion}.
Letting $t\to 0$, and applying (\ref{eqn2021-1-25-3}) and (\ref{eqn2021-2-4-1}),  we obtain
\bea
\langle D_\nu [z(\gamma'^\top, \cdot)], E_i\rangle|_p  &=& \nu\left(z(\gamma'^\top, E_i)\right)|_{p} + z(e_i, e_i) \langle D_\nu E_i, E_i\rangle|_p\\
&=&0
\eea
Thus, by  (\ref{eqn2021-1-25-2}) and (\ref{eqn2021-1-25-3}) again, we have
$$
- \frac{sa}{n-1} \langle D_{e_i}\nu, e_i\rangle = (1+a) \langle D_{\nu}[z(\gamma'^\top, \cdot)], E_i\rangle|_p 
- \frac{sa}{n(n-1)} \langle D_{\nu} E_i, E_i\rangle|_p =0.
$$
That is, 
$$
\langle D_{e_i}\nu, e_i\rangle = 0.
$$
Hence,  the square norm of the second fundamental form $A$ is given by
$$
|A|^2(p) = \sum_{i=2}^n  \left|\left(D_{e_i}e_i\right)^\perp\right|^2 = \sum_{i=2}^n \langle D_{e_i}\nu, e_i\rangle^2 = 0,
$$
which shows $\Sigma = f^{-1}(a)$ is totally geodesic.

\vspace{.12in}

Finally, from (\ref{eqn2019-1210-1}), we have
$$
{\rm Ric}(\nu, \nu) =   \frac{s}{n}\cdot \frac{1}{1+a} <0.
$$
Furthermore, the stability operator for hypersurfaces with a vanishing second fundamental form clearly becomes
$$
\int_{\Sigma} \left[|\n \vp|^2 - {\rm Ric}(\nu, \nu)\vp^2\right] \ge 0
$$
for any function $\vp$ on $\Sigma$. By Fredholm alternative (cf. \cite{f-s}, Theorem 1), 
there exists a positive function $\vp>0$ on $\Sigma$ satisfying 
$$
\Delta^\Sigma \vp + {\rm Ric}(\nu, \nu)\vp =0.
$$
However, it follows from the maximum principle that
$\vp$ must be a constant since $\Sigma$ is a compact, which is impossible.
A similar argument shows that $f$ has only one maximum point.

\end{proof}

\section{Proof of Theorem~\ref{thm2020-8-22-3}}

In this section, we will prove Theorem~\ref{thm2020-8-22-3}. Let $(M^n, g, f)$ be a nontrivial solution of the CPE satisfying
$z(\n f, X) = 0$ for any vector field $X$ which is orthogonal to $\n f$.
From Theorem~\ref{thm2018-1-20-11}, it suffices to prove that $T=0$ or $(M,g)$ has  harmonic curvature. 
To this end, we introduce a warped product metric involving
$\frac {df}{|\nabla f|}\otimes \frac {df}{|\nabla f|}$ as a fiber metric on each level
set $f^{-1}(c)$. Consider a warped product metric $\bar{g}$ on $M$ by
\bea
\bar{g}= \frac {df}{|\nabla f|}\otimes \frac {df}{|\nabla f|}+|\nabla f|^2 g_{\Sigma},
\label{eqn2019-8-27-1}
\eea
where $g_{\Sigma}$ is the restriction of $g$ to  $\Sigma:= f^{-1}(-1)$. 
Note that, from Theorem~\ref{thm2019-12-23-1}, 
the metric $\bar g$ is smooth on $M$ except, possibly at two points, the maximum and minimum points of $f$. 
Furthermore, applying Morse theory \cite{mil} together with Theorem~\ref{thm2019-12-23-1}, we can see that
 $M$ is homeomorphic to ${\Bbb S}^{n}$, and fiber $f^{-1}(t)$  is topologically ${\Bbb S}^{n-1}$ except the two critical points of $f$.

The following lemma shows that $\nabla f$ is a conformal Killing vector field with respect to the metric $\bar{g}$.

\begin{lem}\label{lemt1} 
Let $(g,f)$ be a nontrivial solution of the CPE on
 an a $n$-dimensional compact manifold $M$ with $z(\n f, X) = 0$ for $X \perp \n f$. 
 Then,
$$\frac 12 {\mathcal L}_{\nabla f}\bar{g} =N(|\nabla f|)\bar{g}= \frac 1n 
(\bar{\tr }f)\, \bar{g}.
$$
Here, $\mathcal L$ denotes the Lie derivative.
\end{lem}
\begin{proof}
Note that, by (\ref{eqn2020-5-7-2}) we have,
$$ 
\frac 12 {\mathcal L}_{\nabla f} g  =D_gdf =(1+f)z -\frac {sf}{n(n-1)}g.
$$
Let $X$ and $Y$ be two vector fields with $X \perp \n f$ and $Y \perp \n f$. By the definition of Lie derivative, 
\bea
\frac 12{\mathcal L}_{\nabla f}(df \otimes df)(X,Y)&=& Ddf(X, \nabla f)df(Y)+df(X)Ddf(Y, \nabla f)\\
&=& 2\left( (1+f)\a -\frac {sf}{n(n-1)}\right) \, df\otimes df(X, Y).
\eea
Therefore, from (\ref{eqn2020-5-7-3}), 
\be 
\frac 12 {\mathcal L}_{\nabla f}\left( \frac {df}{|\nabla f|} \otimes \frac {df}{|\nabla f|} \right) 
= N(|\nabla f|)  \frac {df}{|\nabla f|} \otimes \frac {df}{|\nabla f|}.\label{eqnt2}
\ee
Since
$$ 
\frac 12 {\mathcal L}_{\nabla f} (|\nabla f|^2 g_{\Sigma})
= \frac 12 \nabla f(|\nabla f|^2) g_\Sigma
= Ddf(\nabla f, \nabla f)g_\Sigma= N(|\nabla f|) |\nabla f|^2 g_\Sigma,
$$
we conclude that
$$ 
\frac 12 {\mathcal L}_{\nabla f}\bar{g}=\bar{D}df= N(|\nabla f|) \bar{g}.
$$
In particular, we have $\bar{\tr} f =n N(|\nabla f|) $.
\end{proof}

\begin{lem} \label{lem2019-6-22-1}
Let $(g,f)$ be a nontrivial solution of the CPE on
 an $n$-dimensional compact manifold $M$ with $z(\n f, X) = 0$ for $X \perp \n f$. 
 Then $T=0$ on $M$.
\end{lem}
\begin{proof}
Let $p, q \in M$ be two points such that $f(p) = \min_M f$ and $f(q) = \max_M f$, respectively, and let $\bar M = M \setminus \{p, q\}$.
 Due to Theorem~\ref{thm2019-12-23-1} and Lemma~\ref{lemt1}, we can apply Tashiro's result \cite{tas} and can see that
  $(\bar M, \bar g)$  is conformally equivalent
to ${\Bbb S}^n\setminus \{\bar p, \bar q\}$, where $\bar p$ and $\bar q$ are the points in ${\Bbb S}^n$  corresponding to $p$ and $q$, respectively.  
In particular, by Theorem 1 in \cite{bgv}, 
the fiber space $(\Sigma, g|_{\Sigma})$ is a space of constant curvature.
Thus, 
$$
(\Sigma, g|_\Sigma) \equiv ({\Bbb S}^{n-1}, \, r \cdot g_{{\Bbb S}^{n-1}}),
$$
where $r>0$ is a positive constant and $g_{{\Bbb S}^{n-1}}$ is a round metric.

Now, replacing $\Sigma = f^{-1}(-1)$ by $\Sigma_t:= f^{-1}(t)$ in 
(\ref{eqn2019-8-27-1}), it can be easily concluded that the warped product metric $\bar g_t$ also
satisfies Lemma~\ref{lemt1}, and hence, the same argument mentioned above 
shows that,
for any level hypersurface $\Sigma_t:= f^{-1}(t)$,
$$
(\Sigma_t, g|_{\Sigma_t}) \equiv ({\Bbb S}^{n-1}, r(t) \cdot  g_{{\Bbb S}^{n-1}}).
$$
Therefore, the original metric $g$ can also be written as
\be
g= \frac {df}{|\nabla f|}\otimes \frac {df}{|\nabla f|}+ b(f)^2 g_\Sigma,\label{eqn2019-12-23-1}
\ee
where $b (f)>0$ is a positive function depending only on $f$. 
From (\ref{eqnt2}) and the following identity
$$ 
\frac 12 {\mathcal L}_{\nabla f}(b^2 g_\Sigma)=b\langle \nabla f, \nabla b\rangle g_\Sigma = b|\nabla f|^2 \frac {db}{df}g_\Sigma,
$$
we obtain
\be 
\frac 12 {\mathcal L}_{\nabla f}g=N(|\nabla f|)\frac {df}{|\nabla f|}\otimes \frac {df}{|\nabla f|}+b|\nabla f|^2 \frac{db}{df}g_\Sigma.\label{eqnt5-1}
\ee
On the contrary, from (\ref{cpe}) together with (\ref{eqn2020-5-7-3}) and (\ref{eqn2019-12-23-1}), we have
\bea
\lefteqn{\frac 12 {\mathcal L}_{\nabla f}g = Ddf =(1+f)z -\frac {sf}{n(n-1)}g}\\
&=& N(|\nabla f|) \frac {df}{|\nabla f|}\otimes \frac {df}{|\nabla f|} +(1+f)z-(1+f)\a \frac {df}{|\nabla f|}\otimes \frac {df}{|\nabla f|}  -\frac {sf}{n(n-1)}b^2 g_\Sigma.
\eea
Comparing this to (\ref{eqnt5-1}), we obtain
\be
\left( b|\nabla f|^2 \frac {db}{df} +\frac {sf}{n(n-1)}b^2\right) g_\Sigma=(1+f)\left(z-\a  \frac {df}{|\nabla f|}\otimes \frac {df}{|\nabla f|}\right).\label{eqnt7-1}
\ee
Now, let $\{E_1, E_2, \cdots, E_n\}$ be a local frame with $E_1=N$. Then, we have
$$ 
b|\nabla f|^2 \frac {db}{df}= (1+f)z(E_i, E_i)-\frac {sf}{n(n-1)}b^2
$$
for each $2\leq j\leq n$. Summing up these, we obtain
$$ 
(n-1)b|\nabla f|^2 \frac {db}{df}= -(1+f)\a -\frac {sf}nb^2.
$$
Substituting this into (\ref{eqnt7-1}), we get
\be
 -\frac {\a}{n-1} g_\Sigma= z-\a  \frac {df}{|\nabla f|}\otimes \frac {df}{|\nabla f|}. \label{eqn2020-8-29-1}
 \ee
Replacing $(\Sigma, g_\Sigma)$ by $(\Sigma_t , g_{\Sigma_t})$, we can see that the argument mentioned above is also valid. 
Thus, (\ref{eqn2020-8-29-1}) shows that, on each level hypersurface $f^{-1}(t)$, we have
$$ 
z(E_i, E_j)=-\frac {\a}{n-1}
$$
for $2\leq j\leq n$. Hence,
$$ 
|z|^2=\a^2+\frac {\a^2}{n-1}=\frac n{n-1}\a^2=\frac n{n-1}|i_Nz|^2,
$$
since $z(N, E_i)=0$ for $i\geq 2$. As a result, it follows from 
Lemma~\ref{lem2019-5-28-5} that $T=0.$

\end{proof}

\begin{rem}\label{rem2020-9-28-3}
{\rm 
Let $(g, f)$ be a non-trivial solution of the CPE with $z(\n f, X)  = 0$ for $X \perp \n f$. 
In Appendix, we show the following result. 
\be
\frac{s}{n(n-1)} g  = R_N +z +\frac{1+f}{|\n f|^2}i_{\n f}C- \left(\a - \frac{s}{n(n-1)}\right) \frac{df}{|df|}\otimes \frac{df}{|df|}.\label{eqn2020-9-28-1}
\ee
Here, $R_N$ is defined as follows 
$$
R_N(X, Y) = R(X, N, Y, N)
$$
for any vector field $X$ and $Y$. Let $\e:= \frac{s}{n(n-1)}$ and
$$
h:= R_N +\left(z- \a \frac{df}{|df|}\otimes \frac{df}{|df|}\right) +\frac{1+f}{|\n f|^2}i_{\n f}C.
$$
We can then rewrite the metric $g$ as
\be
g = \frac{1}{\e} h + \frac{df}{|df|}\otimes \frac{df}{|df|}.\label{eqn2020-9-28-10}
\ee
Then, the following can be proved on the set $f^{-1}(-1)$:
\be
{\mathcal L}_{\n f} \left(\frac{h}{|\n f|^2}\right) = 0\label{eqn2020-9-28-2}
\ee
Therefore, we can conclude that $g$ can be expressed as a warped product metric, and that it is,
in fact, equal to the metric $\bar g$ defined at the beginning of Section 4.
Refer to the Appendix for the detailed proofs of (\ref{eqn2020-9-28-1}) and (\ref{eqn2020-9-28-2}).

}
\end{rem}

Combining Theorem~\ref{thm2018-1-20-11} and Lemma~\ref{lem2019-6-22-1}, we obtain the following theorem.

\begin{thm}
Let $(g,f)$ be a nontrivial solution of the CPE on
 an $n$-dimensional compact manifold $M$ with $z(\n f, X) = 0$ for $X \perp \n f$. 
Then $M$ is isometric to a standard sphere ${\Bbb S}^n$.
\end{thm}

\section{CPE with positive isotropic curvature}

In this section, we will prove that if $(g, f)$ is a nontrivial solution of the CPE with positive isotropic curvature, then $M$ is isometric to a standard sphere. In the view of Theorem~\ref{thm2020-8-22-3}, 
it suffices to show that,
$$
z(\n f, X) = 0
$$
for any $X \perp \n f$. To do this, we introduce a $2$-form $\omega$ on $M$ defined as
$$
\omega := df \wedge i_{\nabla f}z
$$ 
by considering $i_{\nabla f} z$ as a $1$-form.

In this section, the dimension of the manifold $M$ is assumed to be $n \ge 4$.
First, we have the following.

 \begin{lem}\label{lem2018-4-30-20}
 We have
 \be
 \omega = (n-1) {\tilde i}_{\n f} T = -(1+f) {\tilde i}_{\n f}C.\label{keyeqn}
 \ee
 \end{lem}
 \begin{proof}
 As in the proof of Lemma~\ref{Cotz}, it follows from the  definition of $T$ that
 $$
 (n-2) T(X. Y, \nabla f)
=\frac {n-2}{n-1}\, df\wedge i_{\nabla f} z(X,Y)=\frac {n-2}{n-1}\, \omega (X,Y)
$$
for vectors $X$ and $Y$. The second equality follows from
Lemma~\ref{lem2019-5-23-10}.
 \end{proof}

 \begin{lem}\label{simple}
 Let  $\{E_1, E_2,  \cdots, E_n\}$ be a local frame with $E_1=N = \frac{\n f}{|\n f|}$.
Then
$$
\omega =0 \quad \mbox{if and only if}\quad \tilde{i}_{\nabla f}C(N, E_j)=0\quad (j \ge 2).
$$
\end{lem}
\begin{proof}
It follows from the definition of $\omega$ that
\be
\omega (E_j, E_k)=0\quad\mbox{for all $j, k \ge 2$},\label{eqn2019-5-26-1}
\ee
which shows, by Lemma~\ref{lem2018-4-30-20}, 
\bea
\tilde{i}_{\nabla f}C(E_j, E_k)=0\label{simple-2}
\eea
for $2\leq j,k\leq n$.  Therefore, it is easy to see that
$$
\omega =0 \quad \mbox{if and only if}\quad \tilde{i}_{\nabla f}C(N, E_j)=0
$$
for $2\leq j\leq n$.
\end{proof}

\vspace{.12in}
Next, we prove that $\omega$ is closed, and when $(M, g)$
has positive isotropic curvature,  $\omega$ is vanishing.

\begin{lem} \label{lem191} 
As a $2$-form, we have the following 
$$  
\tilde{i}_{\nabla f}C= di_{\nabla f}z.
$$
\end{lem}
\begin{proof} 
Choose a local frame $\{E_i\}$ which is normal at a point $p \in M$, 
and let $\{\theta ^i\}$ be its dual coframe so that $d\theta^i\vert_{p}=0$.  
Since $i_{\nabla f}z=\sum_{l,k=1}^n f_l z_{lk}\theta ^k$ with $E_l(f) = f_l$ and
$z(E_l, E_k) = z_{lk}$, by (\ref{eqn2020-5-7-2}), we have
\bea
di_{\n f}z &=& 
\sum_{j,k} \sum_{l} (f_{lj} z_{lk} + f_l z_{lk;j}) \theta^j \wedge \theta^k \\
&=&
\sum_{j<k} \sum_{l} \left\{(f_{lj} z_{lk} -f_{lk} z_{lj}) +  f_l (z_{lk;j} - z_{lj;k})\right\} 
\theta^j \wedge \theta^k \\
&=&
\sum_{j<k} \sum_{l} \left[\left\{\left( (1+f)z_{lj} -\frac{sf\,\delta_{lj} }{n(n-1)}\right) z_{lk}
- \left( (1+f)z_{lk} -\frac{sf\,\delta_{lk}}{n(n-1)} \right) z_{lj}\right\}  \right] \theta^j \wedge \theta^k \\
& &+ \sum_{j<k} \sum_{l} f_l C_{jkl}  \theta^j \wedge \theta^k\\
&=&
 \sum_{j<k} \sum_{l} f_l C_{jkl}  \theta^j \wedge \theta^k =
 {\tilde i}_{\n f}C.
 \eea
\end{proof}

\begin{lem}\label{closedform}
$\o$ is a closed $2$-form, i.e.,
$d\o = 0$.
\end{lem}
\begin{proof}
Choose a local frame $\{E_i\}$ with $E_1 = N = {\n f}/{|\n f|}$, and let $\{\theta^i\}$ be its dual coframe. 
Then, by Lemma~\ref{simple} and Lemma~\ref{lem191} 
$$
di_{\n f}z = 
\sum_{j<k} \sum_{l} f_l C_{jkl}  \theta^j \wedge \theta^k 
= \sum_{j<k}  |\n f| C_{jk1}  \theta^j \wedge \theta^k 
=
 |\n f| \sum_{k=2}^n   C_{1k1}  \theta^1 \wedge \theta^k. 
$$
Thus, by taking the exterior derivative of $\omega$ in (\ref{keyeqn}), we have
$$
d\o = - df \wedge d i_{\n f}z = - |\n f| \theta^1 \wedge
\left( |\n f| \sum_{k=2}^n   C_{1k1}  \theta^1 \wedge \theta^k\right) =0.
$$
\end{proof}

Now, let $\Omega =\{p\in M \, \vert\, \omega_p \neq 0$ on $T_pM\}$. Then, $\Omega$ is an open subset of $M$. 
We start with the following observation.

\begin{lem} \label{compr2} 
Suppose that $\omega_p\neq 0$ at $p\in M$. Then, 
\be
|D\omega |^2(p)\geq |\delta \omega|^2(p).\label{eqn2019-12-25-1}
\ee
\end{lem}
\begin{proof}
First of all, since $\omega_p\neq 0$, we have $df_p\neq 0$ and $f(p)\neq -1$
by definition of $\omega$ and Lemma~\ref{lem2018-4-30-20}. 
Define $A:T_pM\to T_pM$ by $ g(Au, v)=\omega (u, v)$ for any $u, v\in T_pM$.

\vskip .5pc
\noindent
{\bf Assertion 1}: $\nabla f(p)\notin \ker A$.

 Let $\{e_1, e_2, \cdots,  e_n\}$ be an orthonormal  basis on $T_pM$ with 
$e_1=N(p)$. If $\nabla f(p)\in \ker A$, then, 
$$
0=\langle Ae_1, e_j\rangle =\omega(e_1, e_j)=-(1+f)\, \tilde{i}_{\nabla f}C(e_1, e_j)
$$
for $j\geq 2$. This implies that $\omega_p=0$ from Lemma~\ref{simple}, a contradiction. 

\vskip .5pc
\noindent
{\bf Assertion 2}: $\ker A\subset (\nabla f)^{\bot}$.

 Let $u\in \ker A$ so that $\langle Au, v\rangle =0$ for any $v\in T_pM$. 
 Let $\{e_1, e_2, \cdots, e_n\}$ be an orthonormal basis on $T_pM$ with $e_1=N(p)$.
 Then, by Lemma~\ref{lem2018-4-30-20} together with (\ref{eqn2019-5-26-1}), we have
\bea
0&=& \langle Au, e_k\rangle = \omega(u, e_k)= -(1+f)\, \tilde{i}_{\nabla f}C(u, e_k)\\
&=& -(1+f) \, \sum_{j=1}^n \langle u, e_j\rangle \, \tilde{i}_{\nabla f}C(e_j, e_k)= -(1+f) \,  \langle u, e_1\rangle \, \tilde{i}_{\nabla f}C(e_1, e_k)
\eea
for any $1\leq k\leq n$. Since $\o_p \ne 0$,  we have
$\tilde{i}_{\nabla f}C(e_1, e_k) \ne 0$  for some  $k\ge 2$ by  Lemma~\ref{simple}.
So, $\langle u, e_1\rangle =0$, which implies that $\n f(p) \in (\ker A)^\perp$.

\vskip .5pc
\noindent
{\bf Assertion 3}: Let $\{e_1,... , e_n\}$ be an orthonormal  basis on $T_pM$ 
with $e_1=N(p)$. Then,
$$
||Ae_1||=\sup_{\substack{u\in (\ker A)^{\bot}\\ ||u||=1}} ||Au||.
$$
First, by {\bf Assertion 1} and {\bf Assertion 2}, 
$Ae_1 \ne 0$ and $e_1 \in (\ker A)^{\bot}$. 
Since  $ g(Ae_1, e_1) = \omega(e_1, e_1)=0$, $Ae_1$ is orthogonal to $e_1$, we may assume that $e_2= {Ae_1}/{||Ae_1||}$. 
In particular, we have $e_2 \in (\ker A)^\perp$ by the skew-symmetry of $A$.

Le $u\in (\ker A)^{\bot}$ with $||u||=1$.
Since $ \langle Ae_i, e_j\rangle =\omega (e_i, e_j)=0$ for $i,j\geq 2$
by (\ref{eqn2019-5-26-1}), we have
\bea
Au&=& \sum_{j=1}^n\langle Au, e_j\rangle e_j=\sum_{j=1}^n \sum_{i=1}^n \langle u, e_i\rangle \, \langle Ae_i, e_j\rangle e_j\\
&=& \langle u, e_1\rangle Ae_1 -\sum_{i=1}^n\langle u,e_i\rangle \langle e_i, Ae_1\rangle e_1\\
&=& 
\langle u, e_1\rangle Ae_1 - \langle u, Ae_1\rangle e_1\\
&=&
||Ae_1||\langle u, e_1\rangle e_2 - ||Ae_1||\langle u, e_2\rangle e_1.
\eea
Since $||u||^2= \sum_j \langle u, e_j\rangle^2=1,$ we have
$$ 
||Au||^2= ||Ae_1||^2 (\langle u, e_1\rangle ^2 +\langle u, e_2\rangle^2 ) \leq ||Ae_1||^2
$$
and hence {\bf Assertion 3} is satisfied.

Note that, for $u=e_j$ with $j\geq 2$, we also have
\bea
Ae_j = -\langle e_j, Ae_1\rangle e_1 \quad\mbox{and so}\quad
 ||Ae_j ||\leq ||Ae_1||.\label{compr}
\eea

\vspace{.102in}
Now, let us show  the inequality (\ref{eqn2019-12-25-1}).
Applying the argument in the proof of Lemma 2.2 from \cite{pz}, we may conclude
 that there exists a local frame $\{E_1,\cdots, E_{2m}, \cdots, E_n\}$  around a point $p \in M$ such that
$$
\omega= \sum_{i=1}^m a_i \theta^{2i-1} \wedge \theta^{2i},
$$
where $\{ \theta^1,..,\theta^n\}$ is the dual coframe of $\{E_i\}_{i=1}^n$. 
In particular, since $\nabla f(p)\notin \ker A$, we may choose 
$E_1=N=\nabla f/||\nabla f||$. 
Since $\omega(E_j,E_k)= 0$ for $j,k\geq 2$ by Lemma~\ref{simple},  we have 
$$
\omega= u \, \theta^1 \wedge \theta^2
$$
for a local smooth function $u$.
Thus, it is easy to see (cf. \cite{pl1}, p.25) that
$$\delta \omega=E_2(u)\theta^1 -E_1(u)\theta^2$$
and so,
$$|\delta \omega|^2 =(E_1(u))^2+(E_2(u))^2.$$
On the contrary, from $\omega= df \wedge i_{\nabla f}z
= u\theta^1 \wedge \theta^2$, we have
$$
\omega (E_1, E_j)= |\nabla f|z(\nabla f, E_j)=0\quad (j \ge 3).
$$
Thus, 
\bea
D_{E_1}\omega (E_1,E_2)&=&E_1(\omega(E_1,E_2))-\omega(D_{E_1}E_1, E_2)- \omega(E_1, D_{E_1}E_2)\\
&=&
E_1(u) -\sum_{j=3}^n\langle D_{E_1}E_1, E_j\rangle\, \omega (E_j,E_2)
- \sum_{j=3}^n \langle D_{E_1}E_2, E_j\rangle\, \omega(E_1, E_j)\\
&=& E_1(u). 
\eea
Similarly, since $D_{E_1}\omega (E_1,E_2) = E_2(u)$,  we may conclude that
$$
|\delta \omega|^2=(E_1(u))^2+(E_2(u))^2 =|D_{E_1}\omega(E_1,E_2)|^2
+|D_{E_2}\omega (E_1,E_2)|^2\leq |D\omega|^2. 
$$
\end{proof}

Using Lemma~\ref{compr2} and Bochner-Weitzenb\"ock formula for $2$-forms, we can prove
 the following structural property for the CPE with positive isotropic curvature.

\begin{thm}\label{thm2018-5-16-2}
Let $(g,f)$ be  a nontrivial solution of (\ref{cpe}) on a compact manifold $M$ of dimension $n \ge 4$. 
If $(M,g)$ has positive isotropic curvature, then the $2$-form $\o = df\wedge i_{\n f}z$ is vanishing.
\end{thm}
\begin{proof}
It suffices to prove that $\Omega =\emptyset$, where  $\Omega =\{p\in M \, \vert\, \omega_p \neq 0$ on $T_pM\}$.  
Suppose, on the contrary,  $\Omega \neq \emptyset$. For $p\in \Omega$, let $\Omega_0$ be a
 connected component of $\Omega$ containing $p$. 
 Note that $ \tr \omega =  -d \delta \omega$  by Lemma~\ref{closedform}.
It follows from the Bochner-Weitzenb\"ock formula for $2$-forms 
(cf. \cite{pl1}, \cite{wu}) that 
\be \frac 12 \tr |\omega |^2 =\langle \tr \omega , \omega \rangle +|D\omega|^2 +\langle E(\omega), \omega\rangle,
\label{bwf}\ee
where $E(\omega)$ is a (local) $2$-form containing isotropic curvature terms as its coefficients.
In particular, if $(M,g)$ has positive isotropic curvature, following the Proposition 2.3 in \cite{pz}
(cf.  \cite{m-w}, \cite{sea}) we have
\be \langle E(\omega), \omega\rangle >0. \label{piccu}
\ee
Therefore, integrating (\ref{bwf}) over $\Omega_0$, we obtain
$$ \frac 12 \int_{\Omega_0} \tr |\omega|^2 =\int_{\Omega_0} (\langle \tr \omega, \omega \rangle + |D\omega |^2 )+\int_{\Omega_0} \langle E(\omega), \omega\rangle.
$$
Since $\omega =0$ on the boundary $\partial \Omega_0$ and $\omega$ is a closed form by Lemma~\ref{closedform}, we have
$$ 0=\int_{\Omega_0} |D\omega |^2 -|\delta \omega|^2 +\langle E(\omega), \omega\rangle.
$$
However, by Lemma~\ref{compr2} and the inequality (\ref{piccu}), the above equation is impossible if $\omega$ is nontrivial. Hence, we may conclude that $\omega =0$, or $\Omega_0=\emptyset$.
\end{proof}

\begin{thm} \label{thm1-1}
Let $M$ be an $n$-dimensional compact  smooth manifold with $n\geq 4$.  
If $(g,f)$ is a nontrivial solution of (\ref{cpe}) and $(M, g)$ has positive isotropic curvature, 
then $(M,g)$ is isometric to a standard sphere.
\end{thm}

\section{Appendix}
In the Appendix, we first claim that $-1$ is a regular value of the potential function $f$ 
when ${\rm Ric}(\n f, X) = 0$ for any vector field $X$ orthogonal to $\n f$ unless $(M, g)$ is Einstein.
Second, we prove (\ref{eqn2020-9-28-1}) and (\ref{eqn2020-9-28-2}) in Remark~\ref{rem2020-9-28-3}.

\subsection{regularity of $f$ on the set $f^{-1}(-1)$}

Define
$$
\vp:= \frac{1}{2}|\n f|^2 + \frac{s}{2n(n-1)}f^2.
$$
Then we have,
$$
N(\vp) = (1+f)\a |\n f|
$$
and 
$$
\Delta \vp = (1+f)^2 |z|^2 + |\n f|^2 \a.
$$
Therefore, 
\be
\Delta \vp - \frac{|\n f|}{1+f} N(\vp) = (1+f)^2 |z|^2 \ge 0. \label{eqn2020-10-8-1}
\ee
By the maximum principle, 
$$
\max_{f\le -1-\e} \vp = \max_{f=-1-\e} \vp
$$
and 
$$
\max_{f \ge -1+\e}\vp = \max_{f = -1+\e}\vp
$$
for  sufficiently small $\e>0$. So, letting $\e \to 0$, we have
\be
\max_M \vp = \max_{f=-1} \vp.\label{eqn2020-10-8-2}
\ee
Now, as mentioned in Introduction, if $\min_M f \ge -1$, then $(M, g)$ is Einstein,
 and so we may assume that $\min_M f <-1$. Let $\min_{x\in M} f(x) = f(x_0)$. Then
for any point $p \in f^{-1}(-1)$, we have
$$
\vp(p) = \frac{1}{2}|\n f|^2(p) + \frac{s}{2n(n-1)} \ge  \vp(x_0) = \frac{s}{2n(n-1)}[f(x_0)]^2 >
\frac{s}{2n(n-1)},
$$
which proves our claim.

\subsection{warped product metric}
\begin{lem}\label{lem7}
Suppose that $z(\n f, X) = 0$ for $X \perp \n f$. 
Then, 
\begin{itemize}
\item[(1)] for vectors $X, Y$ orthogonal to $\n f$,
$$
i_{\n f}T(X, Y) = \frac{|\n f|^2}{n-2}\left(z+\frac{\a}{n-1}g\right)(X, Y).
$$
\item[(2)] $i_{\n f}T(\n f, X) = i_{\n f}T(X, \n f) = 0$ for any vector $X$.
\end{itemize}
\end{lem}
\begin{proof}
If $\o= 0$, then $i_{\n f}z = \a df$ and hence
$$
T = \frac{1}{n-2} df\wedge \left(z+\frac{\a}{n-1}g\right).
$$
\end{proof}

 It follows from Lemma~\ref{lem2019-5-23-10} that
 \be
 -|\n f|^2 \mathcal W_N = (1+f)i_{\n f}C + (n-1)i_{\n f}T.\label{eqn2019-5-28-1}
 \ee
For the Weyl curvature tensor $\mathcal W$, $\mathcal W_N$ can be similarly defined as $R_N$.

\begin{lem} \label{lem2018-8-15-4}
Let $(g, f)$ be a non-trivial solution of the CPE with $\o = 0$. Then,
$$
\frac{s}{n(n-1)} g  = R_N +z +\frac{1+f}{|\n f|^2}i_{\n f}C- \left(\a - \frac{s}{n(n-1)}\right) \frac{df}{|df|}\otimes \frac{df}{|df|}.
$$
\end{lem}
\begin{proof}
 Let
 $$
 \Phi:=  \frac{s}{n(n-1)}g  - z - \frac{1+f}{|\n f|^2}i_{\n f}C.
 $$
 For vector fields $X, Y$ with $X \perp \n f$ and $ Y \perp \n f$, 
from the curvature decomposition
$$
R = \frac{s}{2n(n-1)}g\owedge g + \frac{1}{n-2}z\owedge g + \mathcal W, 
$$
we obtain
 \bea
 R_N(X, Y) = \frac{s}{n(n-1)}g(X, Y) + \frac{1}{n-2} z(X, Y) + \frac{\a}{n-2} g(X, Y)
 +\mathcal W_N(X, Y).
 \eea
Since, by Lemma~\ref{lem7} together with (\ref{eqn2019-5-28-1}),
 \bea
 \mathcal W_N(X, Y) =
 - \frac{1+f}{|\n f|^2}i_{\n f}C(X, Y) - \frac{n-1}{n-2}z(X, Y) - \frac{\a}{n-2}g(X, Y),
 \eea
we have
\be
R_N(X, Y) = \Phi(X, Y).\label{eqn2019-5-28-2}
\ee
Now, let $X$ and $Y$  be arbitrary tangent vector fields. Then, $X$ and $Y$ can be decomposed into 
  $$
  X = X_1 + \langle X, N\rangle N,\quad Y = Y_1 +\langle Y, N\rangle N
  $$
  with $\langle X_1, N\rangle =0 =\langle Y_1, N\rangle.$
  Thus,
  \bea
  R_N(X, Y)    &=& 
  R_N(X_1, Y_1)= \Phi (X_1, Y_1)\\
&=&
\Phi(X, Y) - \langle X, N\rangle  \langle Y, N\rangle  \Phi(N, N)\\
 &=&
\Phi(X, Y) +   \left(\a -\frac{s}{n(n-1)} \right)  \frac{df}{|df|}\otimes \frac{df}{|df|}(X, Y).
   \eea
   
\end{proof}

\begin{lem}
On the set $f^{-1}(-1)$,  we have
\bea
{\mathcal L}_{\n f} \left(\frac{h}{|\n f|^2}\right) = 0.
\eea
Here, recall that $h:= R_N +\left(z- \a \frac{df}{|df|}\otimes \frac{df}{|df|}\right) 
+\frac{1+f}{|\n f|^2}i_{\n f}C.$
\end{lem}
\begin{proof}
First, from CPE, we have
$$
\frac{1}{2}{\mathcal L}_{\n f} g = Ddf = (1+f)z - fh - \e f  \frac{df}{|df|}\otimes \frac{df}{|df|}
$$
with $\epsilon = \frac{s}{n(n-1)}$. It follows from (\ref{eqn2020-5-7-3}), (\ref{eqnt2}) and (\ref{eqn2020-9-28-10}) that
\bea 
\frac 1{2\e} {\mathcal L}_{\nabla f} h = (1+f) z - f h - (1+f)\a  \frac {df}{|\nabla f|} \otimes \frac {df}{|\nabla f|}.
\eea
Thus, 
\bea
\frac{1}{2}{\mathcal L}_{\n f} \left(\frac{h}{|\n f|^2}\right)  
&=&
\frac{1}{2}\n f (|\n f|^{-2}) h + \frac{1}{2|\n f|^2} {\mathcal L}_{\n f} h\\
&=&
-N(|\n f|) \frac{h}{|\n f|^2} + \frac{1}{|\n f|^2} \frac{1}{2}{\mathcal L}_{\n f} h.
\eea
In particular, on the set $f^{-1}(-1)$, we have
$$
\frac{1}{2}{\mathcal L}_{\n f} \left(\frac{h}{|\n f|^2}\right)  = 0.
$$
\end{proof}



\end{document}